\documentclass[11pt, reqno]{amsart}

\usepackage{amsmath, amsthm, amscd, amsfonts, amssymb, graphicx, color, enumerate, xcolor, mathrsfs, latexsym, algorithm, algorithmic, float, tikz}

\usepackage[bookmarksnumbered, colorlinks, plainpages]{hyperref}
\hypersetup{colorlinks=true, linkcolor=red, anchorcolor=green, citecolor=cyan, urlcolor=red, filecolor=magenta, pdftoolbar=true}

\theoremstyle{plain}
\newtheorem{theorem}{Theorem}[section]
\newtheorem{proposition}[theorem]{Proposition}

\newtheorem{corollary}[theorem]{Corollary}
\newtheorem{conjecture}[theorem]{Conjecture}

\theoremstyle{definition}
\newtheorem*{definition}{Definition}

\newtheorem*{example}{Example}
\newtheorem{question}{Question}
\newtheorem*{problem}{Problem}

\theoremstyle{remark}
\newtheorem*{remark}{Remark}

\newcommand{\Cay}{\mathrm{Cay}}
\newcommand{\pmd}{\mathrm{pmd}}

\begin{document}
\title[Positive matching decompositions of graphs]{Positive matching decompositions of graphs}
\author[M. Farrokhi D. G.]{Mohammad Farrokhi D. G.}
\email{m.farrokhi.d.g@gmail.com,\ farrokhi@iasbs.ac.ir}
\address{Research Center for Basic Sciences and Modern Technologies (RBST), Institute for Advanced Studies in Basic Sciences (IASBS), Zanjan 45137-66731, Iran}

\author[Sh. Gharakhloo]{Shekoofeh Gharakhloo}
\email{gh.shekoofeh@iasbs.ac.ir}
\author[A. A. Yazdan Pour]{{Ali Akbar} {Yazdan Pour}}
\email{yazdan@iasbs.ac.ir}
\address{Department of Mathematics, Institute for Advanced Studies in Basic Sciences (IASBS), Zanjan 45137-66731, Iran}

\subjclass[2010]{Primary 05C70, 05C78, 05C38; Secondary 05C62, 05E40.}
\keywords{Alternating walk, matching, positive matching, LSS-ideal, orthogonal representation}
	
\begin{abstract}
A matching $M$ in a graph $\Gamma$ is positive if $\Gamma$ has a vertex-labeling such that $M$ coincides with the set of edges with positive weights. A positive matching decomposition (pmd) of $\Gamma$ is an edge-partition $M_1,\ldots,M_p$ of $\Gamma$ such that $M_i$ is a positive matching in $\Gamma-M_1\cup\cdots\cup M_{i-1}$, for $i=1,\ldots,p$. The pmds of graphs are used to study algebraic properties of the Lov\'{a}sz-Saks-Schrijver ideals arising from orthogonal representations of graphs. We give a characterization of pmds of graphs in terms of alternating closed walks and apply it to study pmds of various classes of graphs including complete multipartite graphs, (regular) bipartite graphs, cacti, generalized Petersen graphs, etc. We further show that computation of pmds of a graph can be reduced to that of its maximum pendant-free subgraph.
\end{abstract}

\maketitle
%==================================================
\section{introduction}
Let $[n]=\{1, \ldots, n\}$ and $\Gamma = ([n], E)$ be a simple graph. We denote by $\bar{\Gamma}$ the complementary graph of $\Gamma$ with the vertex set $[n]$ and edge set $\bar{E} = \binom{[n]}{2} \setminus E$. An \textit{orthogonal representation} of $\Gamma$ in $\mathbb{R}^d$ assigns to each $i \in [n]$ a vector $u_i \in \mathbb{R}^d$ such that $u^T_i u_j = 0$, whenever $\{i,j\} \in \bar{E}$. Clearly, every graph $\Gamma = ([n], E)$ has a trivial orthogonal representation in $\mathbb{R}^{n}$, in which the vertex $i$ is represented by the standard basis vector $e_i$. More generally, if $\{X_1, \ldots, X_{\chi}\}$ is the coloring classes of $\bar{\Gamma}$, then $\Gamma$ has a orthogonal representation in $\mathbb{R}^{\chi}$, by assigning $x \in X_i$ to $e_i$ \cite[Example 10.5]{L}. 

Orthogonal representations of graphs were introduced by Lov\'{a}sz in 1979 \cite{Lo}, and it is shown that they are intimately related to important combinatorial properties of graphs (see \cite[Chapter 10]{L}). The following remarkable result due to Lov\'{a}sz, Saks, and Schrijver \cite{LSS, LSS-correction} determines under which conditions a graph has a non-degenerate orthogonal representation in the sense that the representation is in general position. Recall that a set of vectors $X$ in $\mathbb{R}^d$ is in \textit{general position} if every $d$-subset of $X$ is linearly independent.

\begin{theorem}
Let $\Gamma$ be a graph of order $n$. The followings are equivalent:
\begin{itemize}
\item[\rm (i)] $\Gamma$ is $(n - d)$-connected;
\item[\rm (ii)] $\Gamma$ has a general position orthogonal representation in $\mathbb{R}^d$.
\end{itemize}
\end{theorem}

The set of all orthogonal representations of a graph $\Gamma=([n], E)$ coincides with the vanishing set in $\mathbb{R}^{n\times d}$ of the ideal
\[L_{\bar{\Gamma}}(d) = \left( x_{i,1}x_{j,1} + \cdots + x_{i,d}x_{j,d} \colon \quad \{i, j\} \in E(\bar{\Gamma})\right),\]
in the polynomial ring $\mathbb{R}[x_{i,k} \colon \; i = 1,\ldots,n,\; k = 1, \ldots, d]$. The reader may refer to \cite[Chapter 10.6]{L} for some results on the geometry of the variety of orthogonal representations. However, from an algebraic point of view, these ideals can be studied over any field $\mathbb{K}$. Accordingly, in \cite{HMSW} and later in \cite{ac-vw}, the authors study the graded ideal 
\[ L_{\Gamma}^\mathbb{K}(d) = \left( x_{i,1}x_{j,1} + \cdots + x_{i,d}x_{j,d} \colon \quad \{i, j\} \in E({\Gamma})\right),\]
in the polynomial ring $S=\mathbb{K}[x_{i,k} \colon \; i = 1,\ldots,n,\; k = 1, \ldots, d]$. The ideal $L_{\Gamma}^\mathbb{K}(d)$ is called the \textit{Lov\'{a}sz-Saks-Schrijver ideal}, \textit{LSS-ideal} for short, of $\Gamma$ with respect to $\mathbb{K}$. For $d = 1$, the ideal $L_\Gamma^\mathbb{K}(d)$ is a squarefree monomial ideal known as the \textit{edge ideal} of $\Gamma$. The algebraic properties of LSS-ideal, such as being prime, radical, and complete intersection are studied in \cite{HMSW,ac-vw}. Also, in \cite{HMSW} the (minimal) primary decomposition of $L_\Gamma^\mathbb{K}(2)$ has been presented in terms of combinatorics of $\Gamma$. Recall that an ideal $I\subseteq S$ is called \textit{complete intersection} if $I$ is generated by a sequence $a_1,\ldots,a_m$ such that $a_i$ is not a zero-divisor of $S/(a_1,\ldots,a_{i-1})$, for $i=1,\ldots,m$. For algebraic background in this paper, the reader may refer to \cite{rys}.

Let $X=(x_{ij})$ be a generic symmetric matrix by which we mean an square $n\times n$ matrix of variables $x_{ij}$, where $x_{ij}=x_{ji}$. Let $\Gamma=([n],E)$ be a graph and $X_\Gamma^\mathrm{sym}$ denote the matrix obtained from $X$ by replacing $ij$ and $ji$ entries by $0$ if $ij\in E$. Let $S=\mathbb{K}[x_{ij}\colon\ 1\leq i\leq j\leq n]$ and $d$ be a positive integer. Then the ideal $I_d^\mathbb{K}(X_\Gamma^\mathrm{sym})$ of all $d$-minors of $X_\Gamma^\mathrm{sym}$ defines a coordinate hyperplane section
of the generic symmetric determinantal variety. It is known from \cite[Proposition 7.5]{ac-vw} that $I_d^\mathbb{K}(X_\Gamma^\mathrm{sym})$ is radical, prime, or has maximal height if $L_{\Gamma}^\mathbb{K}(d)$ is radical, prime, or complete intersection, respectively, provided that $\mathbb{K}$ has characteristic zero. Similar relations also exists with $I_d^\mathbb{K}(X_\Gamma^\mathrm{sym})$ replaced by ideals defining coordinate sections of determinantal and Pfaffian ideals (see \cite{ac-vw} for further details). In \cite[Section 5]{ac-vw}, the authors show that for large enough $d$, the ideal $L_\Gamma^\mathbb{K}(d)$ is prime and complete intersection. Subsequently, $I_d^\mathbb{K}(X_\Gamma^\mathrm{sym})$ is prime and has maximal height. For this purpose, the authors define a graph-theoretical invariant $\pmd(\Gamma) \in \mathbb{N}$, called the positive matching decomposition number of $\Gamma$, and connect this number to algebraic properties of LSS-ideals. Remind that a \textit{matching} in a graph is a set of disjoint edges.

\begin{definition}[{\cite[Definition 5.1]{ac-vw}}]
Given a graph $\Gamma = (V, E)$, a \textit{positive matching} of $\Gamma$ is a matching $M$ of $\Gamma$ such that there exists a weight function $w \colon V \to \mathbb{R}$ satisfying:
\begin{equation}
\begin{split}
& \sum\limits_{v \in e} w(i) >0, \quad\text{ if } e \in M, \\
& \sum\limits_{v \in e} w(i) <0, \quad\text{ if } e \notin M.
\end{split}
\end{equation}
\end{definition}
%--------------------------------------------------
\begin{remark}
Notice that in the definition of positive matchings we may replace real labels with integers by multiplying all labels by a large enough positive integer and replacing new labels with their integer parts. 
\end{remark}
%--------------------------------------------------
\begin{definition}[{\cite[Definition 5.3]{ac-vw}}]
Let $\Gamma = (V, E)$ be a graph. A \textit{positive matching decomposition} (or \textit{pmd}) of $\Gamma$ is a partition $E = \bigcup_{i=1}^p E_i$ of $E$ into pairwise disjoint subsets such that $E_i$ is a positive matching of $(V, E \setminus \bigcup_{j=1}^{i-1} E_j)$, for $i = 1, \ldots, p$. The $E_i$ are called the \textit{parts} of the pmd. The smallest $p$ for which $ \Gamma $ admits a pmd with $ p $ parts will be denoted by $ \pmd(\Gamma) $.
\end{definition}

In what follows, ``the pmd'' of a graph $\Gamma$ denotes the minimum size of all positive matching decompositions of $\Gamma$, namely $\pmd(\Gamma)$, while ``a pmd'' of $\Gamma$ refers to any positive matching decomposition of $\Gamma$. Also, it is evident from the definition that a matching $M$ of a graph $\Gamma$ is positive in $\Gamma$ if and only if it is positive in $\Gamma[M]$, the subgraph of $\Gamma$ induced by vertices of $M$.

The following theorem establishes a nice connection between the pmd of a graph with algebraic properties of the corresponding LSS-ideal.
%--------------------------------------------------
\begin{theorem}[{\cite[Theorem 1.3]{ac-vw}}]\label{algebraic properties of pmd}
Let $\Gamma$ be a graph. Then for $d \geq \pmd(\Gamma)$ the ideal $L_\Gamma^\mathbb{K}(d)$ is a radical complete intersection ideal. In particular, $L_\Gamma^\mathbb{K}(d)$ is prime if $d \geq \pmd(\Gamma)+1$.
\end{theorem}

To become more apparent, in \cite{ac-vw} the authors evaluates the pmd of some well-known families of graphs as follows:
%--------------------------------------------------
\begin{proposition}[{\cite[Lemma 5.4]{ac-vw}}]\label{general bounds for pmds}
Let $\Gamma$ be a graph of order $n\geq2$. Then
\begin{itemize}
\item[\rm (i)]$\pmd(\Gamma) \leq 2 n - 3$;
\item[\rm (ii)]$\pmd(\Gamma) \leq n - 1$ if $\Gamma$ is a bipartite graph;
\item[\rm (iii)]$\pmd(\Gamma) \geq \Delta(\Gamma)$ with equality if $\Gamma$ is a forest.
\end{itemize}
\end{proposition}

Since $\pmd(\Gamma) =\Delta(\Gamma)$, for any forest $\Gamma$, it follows from Theorem \ref{algebraic properties of pmd} that $L_\Gamma^\mathbb{K}(d)$ is a radical complete intersection (resp. prime), if $d \geq \Delta(\Gamma)$ (resp. $d \geq \Delta(\Gamma)+1$). The authors in \cite{ac-vw} find a complete picture in the case of forests.
%--------------------------------------------------
\begin{proposition}[{\cite[Theorem 1.4]{ac-vw}}]
If $\Gamma$ is a forest with maximum valency $\Delta(\Gamma)$, then:
\begin{itemize}
\item[\rm (i)] $L_\Gamma^\mathbb{K}(d)$ is radical for all $d$;
\item[\rm (ii)] $L_\Gamma^\mathbb{K}(d)$ is a complete intersection if and only if $d \geq \Delta(\Gamma)$;
\item[\rm (iii)] $L_\Gamma^\mathbb{K}(d)$ is prime if and only if $d \geq \Delta(\Gamma)+1$.
\end{itemize}
\end{proposition}

Let $\mathbb{K}$ be an algebraically closed field and $S=\mathbb{K}[x_1,\ldots,x_n]$. It is known that the zero-set of an ideal $I\subseteq S$ is irreducible, in the sense of Zariski topology, if and only if $I$ is a prime ideal. In particular, the variety of orthogonal representations of a graph is irreducible if and only if the corresponding LSS-ideal (with respect to $\mathbb{K}$) is prime. Hence, motivated by Theorem~\ref{algebraic properties of pmd}, the pmd of graphs turns into an important graph-theoretical invariant that can be applied in both algebra and geometry. In addition, as we show in the next section, the pmd of graphs has connections to some other properties of graphs as well. The aim of this paper is to study the pmd of graphs from graph-theoretical point of view and compute it for various classes of graphs. Our results in conjunction with Theorem~\ref{algebraic properties of pmd} enable us to obtain some algebraic properties of the LSS-ideal of the graphs under consideration. This paper is organized as follows:

In Section \ref{section 2}, we introduce the notion of alternating walks and show that a matching $M$ in a graph $\Gamma$ is positive if and only if the subgraph of $\Gamma$ induced by $M$ has no alternating closed walks with respect to $M$. This is a key theorem and we apply this correspondence in all of our later results. The first application of this result is to show that classifying graphs in terms of their pmds is as hard as classifying graphs with a given maximum valency (see Proposition \ref{subdivision}). However, we can apply a reduction process to simplify the computation of pmds. Indeed, the next main result of this section states that to compute the pmd of a graph $\Gamma$ we need only to compute the pmd of the maximum pendant-free subgraph $\Gamma'$ of $\Gamma$. More precisely, we show that
\[\pmd(\Gamma)=\max\{\pmd(\Gamma'), \Delta(\Gamma)\}.\]

In Section \ref{section 3}, we discuss on the pmd of multipartite graphs. The first result of this section is to give lower and upper bounds for the pmd of complete multipartite graphs (see Theorem \ref{pmd of complete multipartite graphs}). Our result is based on special cases for which we have the exact value of pmd, namely $\pmd(K_n)=2n-3$ and $\pmd(K_{m,n})=m+n-1$. In the rest of this section, we consider bipartite graphs and give two approaches to compute an upper bound for their pmds. The first one yields an upper bound for the pmd of a bipartite graph in terms of its maximum valency, while the second one establishes an algorithm to compute an upper bound for pmd. Our computation shows that the algorithm is efficient in many cases including the particular graphs discussed in this paper.

The last section is devoted to computing the pmd of some further classes of graphs. We obtain the exact value of the pmd of generalized Petersen graphs and M\"{o}bius graphs, and give a, possibly exact, upper bound for the pmd of hypercubes. For those graph-theoretical concepts not mentioned in the paper, we refer the reader to the book \cite{rd}.
%==================================================
\section{Positive matching decompositions vs. alternating closed walks}\label{section 2}
The aim of this section is to give a characterization of positive matching via alternating closed walks. It turns out that a matching $M$ of a graph $\Gamma$ is positive if and only if the subgraph of $\Gamma$ induced by $M$ has no alternating closed walks with respect to $M$. Moreover, in order to compute the pmd of $\Gamma$, one may consider the maximum pendant-free subgraph of $\Gamma$ (Theorem \ref{pmd of antler graphs}). Recall that a vertex of a graph is \textit{pendant} if it has valency one. Also, an edge is pendant if one of its end vertices is a pendant vertex. 
%--------------------------------------------------
\begin{definition}
An \textit{alternating walk} in a graph $\Gamma$ with respect to a matching $M$ is a walk whose edges alternate between edges of $M$ and $M^c:=E(\Gamma)\setminus M$.
\end{definition}

It is evident from the definition that every alternating closed walk has an even number of edges. The following theorem plays a crucial role in the sequel.
%--------------------------------------------------
\begin{theorem}\label{positivity of matchings}
Let $M$ be a matching in a graph $\Gamma$. The following conditions are equivalent:
\begin{itemize}
\item[\rm (i)] $M$ is positive;
\item[\rm (ii)] The subgraph induced by $M$ does not contain any alternating closed walk;
\item[\rm (iii)] The subgraph induced by every subset of $M$ contains a pendant edge belonging to $M$;
\item[\rm (iv)] There exists an ordering of $M$ as $M=\{e_1,\ldots,e_n\}$ such that $e_i$ is a pendant edge in the subgraph induced by $\{e_1,\ldots,e_i\}$, for $i=1,\ldots,n$.
\end{itemize}
\end{theorem}
\begin{proof}
First we show the equivalence of (i) and (ii).

(i)$\Rightarrow$(ii) Suppose on the contrary that $\Gamma[M]$ contains an alternating closed walk, say $u_1,v_1,\ldots,u_n,v_n,u_{n+1}=u_1$, where $u_iv_i\in M$ and $v_iu_{i+1}\notin M$. Since $M$ is positive, there exists a labeling $\rho \colon V(M)\longrightarrow\mathbb{Z}$ such that $\rho(u_i)+\rho(v_i)>0$ and $\rho(v_i)+\rho(u_{i+1})<0$, for all $i=1,\ldots,n$. Let $x_i=|\rho(u_i)|$ and $y_i=|\rho(v_i)|$, for $i=1,\ldots,n$. Without loss of generality, we can assume that $\rho(u_1)>0$ as $\rho(u_1)+\rho(v_1)>0$. Then $\rho(v_n)<0$, $\rho(u_n)>0$, $\rho(v_{n-1})<0$, $\rho(u_{n-1})>0$, etc. Thus $x_i=\rho(u_i)$ and $y_i=-\rho(v_i)$. The positivity of $M$ shows that 
\[x_1>y_1>x_2>y_2>\cdots>x_n>y_n>x_1,\]
which is a contradiction.

(ii)$\Rightarrow$(iii) Suppose $N\subseteq M$ and $\Gamma[N]$ has no pendants. By assumption, $\Gamma[N]$ has no alternating closed walks. Consider a maximal alternating path $P$ in $\Gamma[N]$, say
\[v_1,v_2,\ldots,v_{n-1},v_n.\]
If $v_1v_2\notin N$, then as $P$ is maximal, there must exist an integer $3\leq i\leq n$ such that $v_1v_i\in N$. Then either $i=n$ and $v_{n-1}v_n\notin N$ so that $v_1,v_2,\ldots,v_n,v_1$ is an alternating cycle, or $v_i$ lies on two edges of $N$, a contradiction. Thus $v_1v_2\in N$, and similarly $v_{n-1}v_n\in N$. As $v_1,v_n$ are not pendants and $P$ is maximal, there exists $1\leq i,j\leq n$ such that $v_1\sim v_i$ and $v_n\sim v_j$. Then either $v_1,v_2,\ldots,v_i,v_1$, or $v_j,v_{j+1},\ldots,v_n,v_j$, or
\[v_1,v_2,\ldots,v_n,v_j,v_{j-1},\ldots,v_i,v_1\]
if $i<j$ (see Figure \ref{figure: i<j}), or
\[v_1,v_2,\ldots,v_j,v_n,v_{n-1},\ldots,v_i,v_1\]
if $i>j$ (see Figure \ref{figure: i>j}), gives rise to an alternating closed walk, contradicting the assumption. 

(iii)$\Rightarrow$(i) Let $u$ be a pendant in $\Gamma[M]$. Then there exists $v\in V(M)$ such that $uv\in M$. Since $\Gamma-u$ satisfies the hypothesis, an inductive argument shows that $M\setminus\{uv\}$ is positive in $\Gamma-u$. Hence, there exists a labeling 
\[\rho \colon V(\Gamma[M]-u)\longrightarrow\mathbb{Z}\]
such that $\rho(x)+\rho(y)>0$ if $xy\in M$ and $\rho(x)+\rho(y)<0$ if $xy\in E(\Gamma[M]-u)\setminus M$. We extend $\rho$ to a map 
\[\rho^* \colon V(\Gamma[M])\longrightarrow\mathbb{Z}\]
by $\rho^*(u)>-\rho(v)$ and $\rho^*(x)=\rho(x)$ for all $x\in V(\Gamma[M]-u)$, establishing the positivity of $M$ (compare \cite[Lemma 5.2(2)]{ac-vw}).

The equivalence of (iii) and (iv) is obvious. The proof is complete.
\end{proof}

\begin{figure}
\centering
\begin{minipage}{.45\textwidth}
\begin{tikzpicture}[scale=0.5]
\node [draw, circle, fill=white, inner sep=1pt, label=below:\tiny{$1$}] (1) at (0, 0) {};
\node [draw, circle, fill=white, inner sep=1pt, label=below:\tiny{$2$}] (2) at (1.5, 0) {};
\node [draw, circle, fill=white, inner sep=1pt, label=below:\tiny{$3$}] (3) at (3, 0) {};
\node [draw, circle, fill=white, inner sep=1pt, label=below:\tiny{$i$}] (i) at (5, 0) {};
\node [draw, circle, fill=white, inner sep=1pt, label=below:\tiny{$j$}] (j) at (7, 0) {};
\node [draw, circle, fill=white, inner sep=1pt, label=below:\tiny{$n-2$}] (n-2) at (9, 0) {};
\node [draw, circle, fill=white, inner sep=1pt, label=below:\tiny{$n-1$}] (n-1) at (10.5, 0) {};
\node [draw, circle, fill=white, inner sep=1pt, label=below:\tiny{$n$}] (n) at (12, 0) {};
\node () at (0,1.75) {};

\draw [very thick] (1)--(2);
\draw (2)--(3);
\draw [thick, dashed] (3)--(i)--(j)--(n-2);
\draw [thick, dashed] (i) to [out=-45, in=-135] (j);
\draw (n-2)--(n-1);
\draw [very thick] (n-1)--(n);
\draw (1) to [out=45, in=135] (i);
\draw (j) to [out=45, in=135] (n);
\end{tikzpicture}  
\caption{$i<j$}
\label{figure: i<j}
\end{minipage}\hfill
\begin{minipage}{.45\textwidth}
\begin{tikzpicture}[scale=0.5]
\node [draw, circle, fill=white, inner sep=1pt, label=below:\tiny{$1$}] (1) at (0, 0) {};
\node [draw, circle, fill=white, inner sep=1pt, label=below:\tiny{$2$}] (2) at (1.5, 0) {};
\node [draw, circle, fill=white, inner sep=1pt, label=below:\tiny{$3$}] (3) at (3, 0) {};
\node [draw, circle, fill=white, inner sep=1pt, label=below:\tiny{$j$}] (j) at (5, 0) {};
\node [draw, circle, fill=white, inner sep=1pt, label=below:\tiny{$i$}] (i) at (7, 0) {};
\node [draw, circle, fill=white, inner sep=1pt, label=below:\tiny{$n-2$}] (n-2) at (9, 0) {};
\node [draw, circle, fill=white, inner sep=1pt, label=below:\tiny{$n-1$}] (n-1) at (10.5, 0) {};
\node [draw, circle, fill=white, inner sep=1pt, label=below:\tiny{$n$}] (n) at (12, 0) {};
\node () at (0,1.75) {};

\draw [very thick] (1)--(2);
\draw (2)--(3);
\draw [thick, dashed] (3)--(j);
\draw (j) to [out=45, in=135] (n);
\draw [thick, dashed] (i)--(n-2);
\draw (n-2)--(n-1);
\draw [very thick] (n-1)--(n);
\draw (1) to [out=45, in=135] (i);
\end{tikzpicture}
\caption{$i>j$}
\label{figure: i>j}
\end{minipage}
\end{figure}
%--------------------------------------------------
\begin{remark}
Since in bipartite graphs every alternating closed walk with respect to a matching contains an alternating cycle, we may replace alternating closed walk by alternating cycle in the above theorem when studying bipartite graphs.
\end{remark}

Let $P=(V,<)$ be a finite partially ordered set and $L:x_1,\ldots,x_n$ be a \textit{linear extension} of $P$ that is $x_i <_P x_j$ implies $x_i <_L x_j$. A pair $(x_i,x_{i+1})$ is a \textit{jump} in $L$ if $x_i$ and $x_{i+1}$ are incomparable in $P$. The \textit{jump number} $\mathrm{jump}(P)$ of $P$ is the minimum number of jumps in all possible linear extensions of $P$. To any partially ordered set there corresponds a comparability graph $C(P)$ with vertex set $V$ such that $xy$ is an edge if $x$ and $y$ are comparable in $P$. We define the jump number of $C(P)$ to be the jump number of $P$. Note that every bipartite graph is the  comparability graph of some poset. In \cite{gc-mc}, Chaty and Chein show that a maximum positive matching $M$ in a bipartite graph $\Gamma$ satisfies the equation
\[|M|+\mathrm{jump}(\Gamma)=|V|-1.\]
The problem of computing the jump number of a poset is known to be NP-complete even for those arising from bipartite graphs (see \cite{wrp}). Since computing the maximum size of a positive matching in bipartite graphs is equivalent to that of the jump number of the graphs, it turns out that maximizing positive matchings in a bipartite graph is also NP-complete. M\"{u}ller \cite{hm} shows the stronger result that computing the maximum positive matching size in the class of chordal bipartite graphs is also NP-complete, hence, by applying the Chaty and Chein's result, computing the jump number of these graphs is NP-complete as well.

An easy application of the above theorem yields the following results immediately.
%--------------------------------------------------
\begin{example}\label{simple pmds}
We have
\begin{itemize}
\item[(i)]$\pmd(\Gamma)=1$ if and only if $\Gamma$ is a union of some edges and isolated vertices.
\item[(ii)]$\pmd(\Gamma)=2$ if and only if $\Gamma$ is a union of some paths not all of length zero or one.
\item[(iii)]$\pmd(C_n)=3$, for all $n\geq3$.
\item[(iv)]$\pmd(\Gamma)\leq|E(\Gamma)|$ for every graph $\Gamma$ with equality if and only if $\Gamma$ is a union of an star graph and some isolated vertices.
\item[(v)]$\pmd(\Gamma)=\max\{\pmd(S)\colon\quad S\ \text{is a connected component of}\ \Gamma\}$.
\end{itemize}
\end{example}

According to part (v) of the above example, from now on, all graphs under considerations are assumed to be connected.

The above example gives a simple classification of all graphs with pmd at most two. However, as the following proposition shows such a classification is far reaching for graphs with pmd at least three. Recall that a \textit{subdivision} of a graph $\Gamma$ is a graph $\Gamma'$ obtained from $\Gamma$ by replacing edges of $\Gamma$ by paths of length at least one. Notice that $\Delta(\Gamma')=\Delta(\Gamma)$ for every subdivision $\Gamma'$ of a connected graph $\Gamma$ provided that $\Gamma$ has at least three vertices. A \textit{starlike graph} is any subdivision of a star graph.
%--------------------------------------------------
\begin{proposition}\label{subdivision}
Every graph $\Gamma$ with maximum valency at least three has a subdivision $\Gamma'$ satisfying $\pmd(\Gamma')=\Delta(\Gamma')$.
\end{proposition}
\begin{proof}
Let $\Gamma'$ be the graph obtained from $\Gamma$ by replacing every edge $e=uv$ of $\Gamma$ by a path $u,u_e,u'_e,v'_e,v_e,v$ of length $5$. Clearly, the set 
\[M_0:=\{u'_ev'_e \colon\quad e=uv\in E(\Gamma)\}\]
is a matching of $\Gamma'$ whose removal is a union of starlike graphs induced by vertex sets $\{u\}\cup\{u_e,u'_e\colon\ u\in e\in E(\Gamma)\}$ with root $u$, for all vertices $u\in V(\Gamma)$. Let $e_u$ be an edge incident with $u$ in $\Gamma$, for each vertex $u\in V(\Gamma)$. By Theorem \ref{positivity of matchings},
\[M:=\{e_u\colon\quad u\in V(\Gamma)\}\cup M_0\]
is a positive matching whose removal is a union of starlike graphs (of maximum valencies at most $\Delta(\Gamma')-1$) and paths of length two (see Figure \ref{figure: subdivision}). Therefore, 
\[\pmd(\Gamma')\leq1+\pmd(\Gamma'-M)\leq\Delta(\Gamma'),\]
from which the result follows.
\end{proof}

\begin{figure}[h]
\begin{tabular}{ccc}
\begin{tikzpicture}[scale=0.4]
\node [draw, circle, fill=white, inner sep=2pt, label=225:\tiny{$a$}] (00) at (0, 0) {};
\node [draw, circle, fill=white, inner sep=2pt, label=315:\tiny{$b$}] (50) at (10, 0) {};
\node [draw, circle, fill=white, inner sep=2pt, label=45:\tiny{$c$}] (55) at (10, 10) {};
\node [draw, circle, fill=white, inner sep=2pt, label=135:\tiny{$d$}] (05) at (0, 10) {};

\node [label=above:\tiny{$\alpha$}] () at (5,0) {};
\node [label=left:\tiny{$\beta$}] () at (10,5) {};
\node [label=below:\tiny{$\gamma$}] () at (5,10) {};
\node [label=right:\tiny{$\delta$}] () at (0,5) {};
\node [label=315:\tiny{$\epsilon$}] () at (5,5) {};

\draw (00)--(50)--(55)--(05)--(00)--(55);
\end{tikzpicture}  
&&
\begin{tikzpicture}[scale=0.4]
\node [draw, circle, fill=white, inner sep=1pt, label=above:\tiny{$a_\alpha$}] (10) at (2, 0) {};
\node [draw, circle, fill=white, inner sep=1pt, label=above:\tiny{$a'_\alpha$}] (20) at (4, 0) {};
\node [draw, circle, fill=white, inner sep=1pt, label=above:\tiny{$b'_\alpha$}] (30) at (6, 0) {};
\node [draw, circle, fill=white, inner sep=1pt, label=above:\tiny{$b_\alpha$}] (40) at (8, 0) {};
\node [draw, circle, fill=white, inner sep=2pt, label=315:\tiny{$b$}] (50) at (10, 0) {};
\node [draw, circle, fill=white, inner sep=1pt, label=left:\tiny{$b_\beta$}] (51) at (10, 2) {};
\node [draw, circle, fill=white, inner sep=1pt, label=left:\tiny{$b'_\beta$}] (52) at (10, 4) {};
\node [draw, circle, fill=white, inner sep=1pt, label=left:\tiny{$c'_\beta$}] (53) at (10, 6) {};
\node [draw, circle, fill=white, inner sep=1pt, label=left:\tiny{$c_\beta$}] (54) at (10, 8) {};
\node [draw, circle, fill=white, inner sep=2pt, label=45:\tiny{$c$}] (55) at (10, 10) {};
\node [draw, circle, fill=white, inner sep=1pt, label=below:\tiny{$c_\gamma$}] (45) at (8, 10) {};
\node [draw, circle, fill=white, inner sep=1pt, label=below:\tiny{$c'_\gamma$}] (35) at (6, 10) {};
\node [draw, circle, fill=white, inner sep=1pt, label=below:\tiny{$d'_\gamma$}] (25) at (4, 10) {};
\node [draw, circle, fill=white, inner sep=1pt, label=below:\tiny{$f_\gamma$}] (15) at (2, 10) {};
\node [draw, circle, fill=white, inner sep=2pt, label=135:\tiny{$d$}] (05) at (0, 10) {};
\node [draw, circle, fill=white, inner sep=1pt, label=right:\tiny{$d_\delta$}] (04) at (0, 8) {};
\node [draw, circle, fill=white, inner sep=1pt, label=right:\tiny{$d'_\delta$}] (03) at (0, 6) {};
\node [draw, circle, fill=white, inner sep=1pt, label=right:\tiny{$a'_\delta$}] (02) at (0, 4) {};
\node [draw, circle, fill=white, inner sep=1pt, label=right:\tiny{$a_\delta$}] (01) at (0, 2) {};
\node [draw, circle, fill=white, inner sep=2pt, label=225:\tiny{$a$}] (00) at (0, 0) {};
\node [draw, circle, fill=white, inner sep=1pt, label=below:\tiny{$a_\epsilon$}] (11) at (2, 2) {};
\node [draw, circle, fill=white, inner sep=1pt, label=below:\tiny{$a'_\epsilon$}] (22) at (4, 4) {};
\node [draw, circle, fill=white, inner sep=1pt, label=below:\tiny{$c'_\epsilon$}] (33) at (6, 6) {};
\node [draw, circle, fill=white, inner sep=1pt, label=below:\tiny{$c_\epsilon$}] (44) at (8, 8) {};
\node [label=below:\tiny{$e_a$}] () at (1,0) {};
\node [label=right:\tiny{$e_b$}] () at (10,1) {};
\node [label=above:\tiny{$e_c$}] () at (9,10) {};
\node [label=left:\tiny{$e_d$}] () at (0,9) {};

\draw (00)--(10)--(20)--(30)--(40)--(50)--(51)--(52)--(53)--(54)--(55)--(45)--(35)--(25)--(15)--(05)--(04)--(03)--(02)--(01)--(00)--(11)--(22)--(33)--(44)--(55);
\draw [very thick] (20)--(30) (25)--(35) (02)--(03) (52)--(53) (22)--(33);
\draw [very thick] (00)--(10) (50)--(51) (45)--(55) (04)--(05);
\end{tikzpicture}\\
Graph $\Gamma$&&Subdivision graph $\Gamma'$
\end{tabular}
\caption{Positive matching $M=\{a'_\alpha b'_\alpha, b'_\beta c'_\beta, c'_\gamma d'_\gamma, d'_\delta a'_\delta, a'_\epsilon c'_\epsilon, e_a, e_b, e_c, e_d\}$ of $\Gamma'$}
\label{figure: subdivision}
\end{figure}
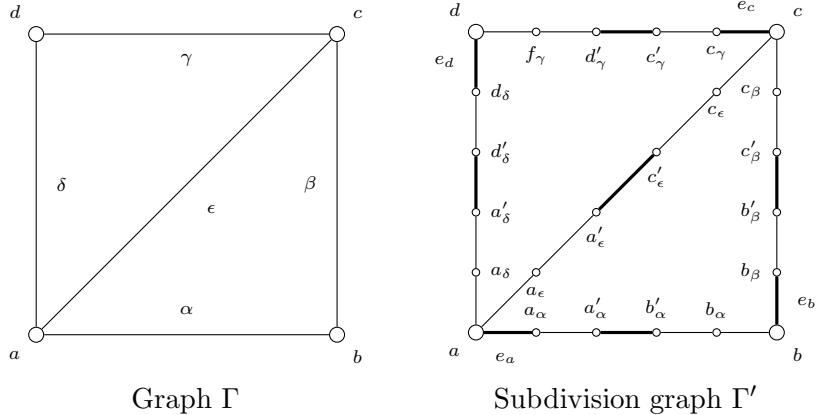

In spite of the fact that the classification of graphs with a given pmd is hopelessly difficult, the next theorem provides us with a reduction in the sense that we can always assume that the graphs under considerations have no pendants.
%--------------------------------------------------
\begin{definition}
A \textit{corona} graph of a graph $\Gamma$ is a graph obtained from $\Gamma$ by attaching some pendants to every vertex of $\Gamma$. An \textit{antler} graph (or \textit{hartshorne} graph) of a graph $\Gamma$ is a graph obtained by successively making coronas of $\Gamma$. Equivalently, an antler graph of $\Gamma$ is a graph obtained from $\Gamma$ by attaching the root of a rooted tree $T_v$ to every vertex $v$ of $\Gamma$.
\end{definition}
%--------------------------------------------------
\begin{theorem}\label{pmd of antler graphs}
Let $\Gamma$ be a graph and $\Gamma'$ be an antler graph of $\Gamma$. Then
\[\pmd(\Gamma') = \max\{\pmd(\Gamma), \Delta(\Gamma')\}.\]
\end{theorem}
\begin{proof}
First assume that $\Gamma'$ is a corona of $\Gamma$. For every vertex $v$ of $\Gamma$ let $X_v := \{e^v_1,\ldots,e^v_{n_v}\}$ be the set of $n_v$ edges of $E(\Gamma') \setminus E(\Gamma)$ incident to $v$, and put $X:=\bigcup_{v\in V(\Gamma)}X_v$. Let $M_1, \ldots, M_p$ be a pmd of $\Gamma$ with $p=\pmd(\Gamma)$, and $i:=\max\{\Delta(\Gamma') - \pmd(\Gamma), 0\}$. Then $\Delta(\Gamma')\leq p+i$ and hence $\deg_\Gamma(v)+n_v\leq p + i$, for all $v\in V(\Gamma)$. Let 
\[M'_1 := M_1 \cup \{e^v_1 \colon\quad v \notin V(M_1)\}\]
and
\[M'_i := M_i \cup \{e^v_j \colon\quad e^v_{j-1} \in M'_1\cup\cdots\cup M'_{i-1},\ e^v_j\notin  M'_1\cup\cdots\cup M'_{i-1},\ \text{and}\ v\notin V(M_i)\},\]
for all $i=2,\ldots,p$. If $M'_{p+1}, \ldots, M'_{p'}$ is a pmd of minimum size for $\Gamma' \setminus (M'_1\cup\cdots\cup M'_p)$, then $M'_1, \ldots, M'_{p'}$ is a pmd of $\Gamma'$ by Theorem \ref{positivity of matchings}(iii). It is evident that $p'=p+i$. This shows that $\pmd(\Gamma')\leq p+i=\max\{p,\Delta(\Gamma')\}$. The reverse inequality is obvious and so $\pmd(\Gamma') = \max\{p, \Delta(\Gamma')\}$.

Now, let $\Gamma'$ be an antler graph of $\Gamma$. By the definition, there exists a sequence of distinct graphs $\Gamma=\Gamma_0,\Gamma_1,\ldots,\Gamma_n=\Gamma'$
such that $\Gamma_i$ is a corona graph of $\Gamma_{i-1}$, for $i=1,\ldots,n$. Then 
\begin{align*}
\pmd(\Gamma')=\pmd(\Gamma_n)&=\max\{\pmd(\Gamma_{n-1}), \Delta(\Gamma_n)\}\\
&=\max\{\pmd(\Gamma_{n-2}), \Delta(\Gamma_n)\}\\
&\hspace{0.2cm}\vdots\\
&=\max\{\pmd(\Gamma_0), \Delta(\Gamma_n)\}=\max\{\pmd(\Gamma), \Delta(\Gamma')\}.
\end{align*}
The proof is complete.
\end{proof}

Since trees are antler graphs of a single vertex, we obtain a new proof of the equality in Proposition \ref{general bounds for pmds}(iii).
%---------------------------------------------------
\begin{corollary}\label{pmd of trees}
If $\Gamma$ is a tree, then $\pmd(\Gamma)=\Delta(\Gamma)$.
\end{corollary}

As a consequence of Theorem \ref{pmd of antler graphs} and Corollary \ref{pmd of trees}, we show that the pmd of a class of graphs, known as cactus graphs, which is a natural generalization of trees is equal to either the maximum valency $\Delta$ or $\Delta+1$. In the following, $S(\Gamma)$ stands for the \textit{subdivision graph} of a graph $\Gamma$ obtained by subdividing every edge of $\Gamma$ by one vertex.
%--------------------------------------------------
\begin{definition}
A \textit{cactus} is a connected graph in which any two cycles have at most one vertex in common.
\end{definition}

A \textit{cut-vertex} in a graph is a vertex whose removal increases the number of connected components. It is evident that every cactus graph that is neither an edge nor a cycle has a cut-vertex.
%--------------------------------------------------
\begin{corollary}
Let $\Gamma$ be a cactus. Then 
\[\Delta(\Gamma)\leq\pmd(\Gamma)\leq\Delta(\Gamma)+1.\]
Moreover, $\pmd(\Gamma)=\Delta(\Gamma)$ if $\Gamma$ is a triangle-free graph which is not a cycle.
\end{corollary}
\begin{proof}
Let $\Delta:=\Delta(\Gamma)$. First we prove that $\pmd(\Gamma)=\Delta$ when $\Gamma$ is a triangle-free cactus graph which is not a cycle. We show that $\Gamma$ has a positive matching $M$ covering all vertices of maximum valency such that $\Gamma - M$ is a tree. From the definition, it follows that $\Gamma$ has a cut-vertex $u$ such that one of the connected components of $\Gamma-u$ is a path $P$. Let $v$ be a neighbor of $u$ in $P$. If $\Gamma-v$ is a cycle, then define $M':=\{uw\}$ for any neighbor $w$ of $u$ in $\Gamma-v$. Otherwise, by an inductive argument, $\Gamma-v$ has a positive matching $M'$ with the required properties. Let $E=\emptyset$ if either $v$ is a pendant or $V(M')\cap V(P)\neq\emptyset$, and $E=\{vw\}$ otherwise, where $w\neq u$ is a neighbor of $v$. Let
\[M:=\begin{cases}
M'\cup E,&u\in V(M'),\\
M'\cup\{uv\},&u\notin V(M').
\end{cases}\]
Then $M$ is a positive matching of $\Gamma$ with the mentioned properties. Now since $\Gamma-M$ is a tree of maximum valency $\Delta-1$, it follows that $\pmd(\Gamma)=\Delta$ by Corollary \ref{pmd of trees}.

Now, we prove that $\pmd(\Gamma)\leq\Delta+1$, for any cactus graph $\Gamma$. If $\Delta\leq3$, then triangles of $\Gamma$ are pairwise disjoint so that a matching $M$ possessing an edge from every triangle of $\Gamma$ is a positive matching in $\Gamma$. Since $\Gamma-M$ is triangle free, it follows that $\pmd(\Gamma)\leq\pmd(\Gamma-M)+1\leq\Delta+1$ provided that $\Gamma$ is not a cycle. If $\Gamma$ is a cycle, then $\pmd(\Gamma)=\Delta+1$, as required. In what follows, we assume that $\Delta>3$. We claim that $\Gamma$ has a positive matching covering all vertices of maximum valency. Suppose the result holds for graphs of smaller order or size. Since $\Gamma$ is a cactus graph, there exists a vertex $u\in V(\Gamma)$ of maximum valency such that $S(\Gamma)-u$ has a connected component $C$ all of whose vertices are of valency less than $\Delta$. Let $X:=V(C)\cap V(\Gamma)$. By assumption, $\Gamma-X$ has a positive matching $M$ covering all vertices of maximum valency. Let $v\in X$ be a neighbor of $u$ in $\Gamma$. Put $M':=M$ if $u\in M$ and $M':=M\cup\{uv\}$ if $u\notin M$. Then $M'$ is a positive matching of $\Gamma$ covering all vertices of maximum valency. Since $\Gamma-M_1$ has fewer edges than $\Gamma$ and $\Delta(\Gamma-M_1)=\Delta-1$, it follows from the hypothesis that 
\[\pmd(\Gamma)\leq\pmd(\Gamma-M_1)+1\leq\Delta+1.\]
The proof is complete.
\end{proof}
%--------------------------------------------------
\begin{problem}
Give a characterization of all those cactus graphs whose pmd coincide with their maximum valency.
\end{problem}
%--------------------------------------------------
\begin{example}
A graph obtained from a cycle by attaching a triangle via an edge to any of its vertices has maximum valency $3$ while its pmd is equal to $4$.
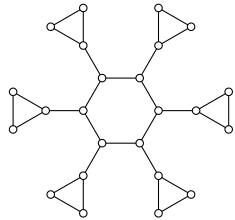
\begin{figure}[h]
\begin{tikzpicture}[scale=0.5]
\foreach \x in {0,...,5}
{	
	\draw[rotate={\x*60}] (0.5, {sqrt(3) / 2})--(1, 0)--(2, 0)--({2 + sqrt(3) / 2}, 0.5)--({2 + sqrt(3) / 2}, -0.5)--(2, 0);
}

\foreach \x in {0,...,5}
{	
	\draw[rotate={\x*60}, fill=white] (1, 0) circle (0.1cm);
	\draw[rotate={\x*60}, fill=white] (2, 0) circle (0.1cm);
	\draw[rotate={\x*60}, fill=white] ({2 + sqrt(3) / 2}, 0.5) circle (0.1cm);
	\draw[rotate={\x*60}, fill=white] ({2 + sqrt(3) / 2}, -0.5) circle (0.1cm);
}
\end{tikzpicture}\\
\caption{The cactus graph obtained from a hexagon}
\label{figure: cactus graph}
\end{figure}
\end{example}

By virtue of Theorem \ref{positivity of matchings}, it is quite easy to recognize positive matchings and even construct them algorithmically. This enables us to write a simple algorithm generating all pmds of a graph and compute their minimum size, that is the pmd of a graph can be determined via a computer program. Since pendants can be removed form the graph when computing the pmd, Theorem \ref{pmd of antler graphs} turns to an improvement of our program, which is particularly efficient when working with sparse graphs (see \cite{mfdg} for the source codes in GAP).
%==================================================
\section{Using independent sets to compute pmds}\label{section 3}
Any graph $\Gamma$ can be partitioned into a family of independent sets, say $\mathcal{I}$, the minimum number of which is usually referred to as the \textit{chromatic number} of the graph. Since $\Gamma$ is a subgraph of the complete multipartite graph $K$ with $\mathcal{I}$ denoting the set of parts, one can observe that $\pmd(K)$ turns to an upper bound for $\pmd(\Gamma)$. In what follows, we first compute the pmd of complete graphs and complete bipartite graphs and apply them to give an upper bound for the pmd of arbitrary complete multipartite graphs. 

We know that for complete graphs $K_n$ ($n\geq2$), $\pmd(K_n)\leq2n-3$ (see Proposition \ref{general bounds for pmds}(i)). Here, we aim to show that the equality holds for all $n\geq2$. 
%--------------------------------------------------
\begin{proposition}\label{pmd of K_n}
We have $\pmd(K_n)=2n-3$, for all $n\geq2$.
\end{proposition}
\begin{proof}
We know from Proposition \ref{general bounds for pmds}(i) that $\pmd(K_n)\leq 2n-3$. Indeed, $M_1,\ldots,M_{2n-3}$ defined as 
\[M_i:=\{u_ru_s\colon\quad r+s=i+2\}\]
is a pmd of $K_n$ assuming that $V(\Gamma)=\{u_1,\ldots,u_n\}$. Now, we prove the lower bound $\pmd(K_n)\geq 2n-3$ by induction on $n$. Suppose $\pmd(K_{n-1})=2(n-1)-3$. Let $M_1,\ldots,M_p$ denote a pmd of $K_n$, and let $t$ be the first number satisfying $|M_t|>1$. It is evident that $t\geq3$. By Theorem \ref{positivity of matchings}(iii), there exist edges $uv,u'v'\in M_t$ such that $u'$ is not adjacent to $u,v$, that is $uu',vu'\in M_1\cup\cdots\cup M_{t-1}$, say $uu'\in M_i$ and $vu'\in M_j$ for some $1\leq i<j<t$. Reordering $M_1,\ldots,M_{t-1}$ if required, we may assume that $i=1$ and $j=2$. Now, since the subgraph induced by $V(K_n)\setminus\{u'\}$ is a complete graph, it follows that
\[\pmd(K_n)=2+\pmd(K_n-\{uu',vu'\})\geq2+\pmd(K_{n-1})=2n-3,\]
as required.
\end{proof}
%--------------------------------------------------
\begin{corollary}
A graph $\Gamma$ with $n\geq2$ vertices is complete if and only if $\pmd(\Gamma)=2n-3$.
\end{corollary}

In \cite[Theorem 8.6(4)]{ac-vw} the authors use pure algebraic techniques to show that the upper bound in Proposition \ref{general bounds for pmds}(ii) is attained for complete bipartite graphs. Here we give a simple graph theoretical proof of this result. To be self contained, we give the proof of the upper bound as well.
%--------------------------------------------------
\begin{proposition}\label{pmd of K_m,n}
We have $\pmd(K_{m,n})=m+n-1$, for all $m,n\geq1$.
\end{proposition}
\begin{proof}
Let $\{u_1,\ldots,u_m\}$ and $\{v_1,\ldots,v_n\}$ be the bipartition of $\Gamma:=K_{m,n}$. For $1\leq k<m+n$, put
\[M_k=\{u_iv_j \colon\quad i+j=k+1,\; 1\leq i\leq m,\; 1\leq j\leq n\}.\]
A simple verification shows that $M_1,\ldots,M_{m+n-1}$ is a pmd of $\Gamma$ so that $\pmd(\Gamma)\leq m+n-1$. To prove $\pmd(\Gamma)\geq m+n-1$, let $M_1,\ldots,M_p$ be a pmd of $\Gamma$ and $i,j\in\{1,\ldots,m\}$ be such that $u_i,u_j$ are not occurred in $M_1,\ldots,M_{k-1}$, for some $1\leq k\leq p$. Then at most one of $u_i$ and $u_j$ can occur in $M_k$, for otherwise we get an alternating square in $M_k$ contradicting the positivity of $M_k$ (Theorem \ref{positivity of matchings}). Hence, if $s$ is the least number for which all $u_1,\ldots,u_m$ occur in $M_1\cup\cdots\cup M_s$, then $s\geq m$. Let $u_i$ be a vertex occurred only once in $M_1\cup\cdots\cup M_s$. As $u_i$ is adjacent to $n-1$ vertices in $\Gamma-(M_1\cup\cdots\cup M_s)$, it follows that $\pmd(\Gamma)\geq m+n-1$. Therefore, $\pmd(\Gamma)=m+n-1$.
\end{proof}
%--------------------------------------------------
\begin{remark}
For any non-negative integers $m<n$ there exists a graph $\Gamma_{m,n}$, say $K_{m+1,n}$, such that $\Delta(\Gamma_{m,n})=n$ and $\pmd(\Gamma_{m,n})-\Delta(\Gamma_{m,n})=m$.
\end{remark}

The above remark suggests us to think about the case $m\geq n$. However, all our computations, insights, and examples in this paper leads us to the following conjecture:
%--------------------------------------------------
\begin{conjecture}
Any graph $\Gamma$ satisfies $\pmd(\Gamma)\leq2\Delta(\Gamma)-1$.
\end{conjecture}
%--------------------------------------------------
%\begin{question}
%Is there any graph $\Gamma$ satisfying $\pmd(\Gamma)\geq2\Delta(\Gamma)$?
%\end{question}

In what follows, we give upper and lower bounds for the pmd of complete multipartite graphs. 
%--------------------------------------------------
\begin{theorem}\label{pmd of complete multipartite graphs}
Let $\Gamma=K_{n_1,\ldots,n_m}$ ($n_1\leq\cdots\leq n_m$) be a complete multipartite graph. Then
\[\max\left\{\frac{3}{2}|\Gamma|-n_m-1, |\Gamma|+\frac{m}{2}-2\right\}\leq\pmd(\Gamma)\leq2|\Gamma|-n_{m-1}-n_m-1.\]
\end{theorem}
\begin{proof}
Let $C_1,\ldots,C_m$ be the corresponding parts of $\Gamma$ with $|C_i|=n_i$, for $i=1,\ldots,m$, and $M_1,\ldots,M_p$ be a pmd of $\Gamma$. 

First we show that $\pmd(\Gamma)\geq \frac{3}{2}|\Gamma|-n_m-1$. Let $X_i$ be the set of vertices of $\Gamma$ covered by $M_i$ but not by $M_1\cup\cdots\cup M_{i-1}$. If $u_1,u_2\in X_i$ are such that $u_1u_2\notin E(\Gamma)$, then there must exist distinct vertices $v_1,v_2\in V(\Gamma)$ different from $u_1,u_2$ such that $u_1v_1,u_2v_2\in M_i$. Then $u_1,v_1,u_2,v_2,u_1$ is an alternating cycle with respect to $M_i$ contradicting the positivity of the matching $M_i$. Thus $u_1$ and $u_2$ are adjacent in $\Gamma$. If $X_i$ contains three vertices $u_1,u_2,u_3$, then $u_1,u_2,u_3$ are vertices of a triangle in $\Gamma$. Let $v_1,v_2,v_3$ be such that $u_1v_1,u_2v_2,u_3v_3\in M_i$. Then either $\{u_1,u_2,u_3\}\cap\{v_1,v_2,v_3\}=\emptyset$ or $\{u_1,u_2,u_3\}\cap\{v_1,v_2,v_3\}\neq\emptyset$. In both cases, the subgraph of $\Gamma$ induced by $\{u_1,u_2,u_3,v_1,v_2,v_3\}$ has no pendants, which contradicts Theorem \ref{positivity of matchings}(iii). This contradiction shows that $|X_i|\leq2$. Let $k$ be such that $X_k\neq\emptyset$ but $X_{k+1}=\cdots=X_p=\emptyset$. Since $|X_1|,\ldots,|X_k|\leq2$, we must have $k\geq|\Gamma|/2$. On the other hand, as the vertices in $X_k$ do not appear in $M_1\cup\cdots\cup M_{k-1}$, $\deg_{\Gamma-M_1\cup\cdots\cup M_{k-1}}(u)=\deg_\Gamma(u)$, for all $u\in X_k$. It follows that \[p-(k-1)\geq\deg_\Gamma(u)\geq\delta(\Gamma)=n_1+\cdots+n_{m-1}\]
for any $u\in X_k$, from which the result follows.

Next we show that $\pmd(\Gamma)\geq|\Gamma|+\frac{m}{2}-2$. Similar to the previous case, let $Y_i$ be the set of parts of $\Gamma$ such that for any part $C$ in $Y_i$, $C$ has a vertex covered by $M_i$ but no vertices of $C$ is covered by $M_1\cup\cdots\cup M_{i-1}$. Let $k$ be such that $Y_k\neq\emptyset$ but $Y_{k+1}=\cdots=Y_p=\emptyset$. Clearly, $k\geq m/2$ as $|Y_1|,\ldots,|Y_p|\leq2$. Let $C\in Y_k$ be any part. An argument analogous to the previous paragraph shows that there must exists $k'\geq k-1+|C|$ such that $C$ is covered by $M_1\cup\cdots\cup M_{k'}$ but not by $M_1\cup\cdots\cup M_{k'-1}$. Let $u\in C$ be a vertex that is not covered by $M_1\cup\cdots\cup M_{k'-1}$. As $\deg_{\Gamma-M_1\cup\cdots\cup M_{k'-1}}(u)=\deg_\Gamma(u)$, we must have $p-(k'-1)\geq\deg_\Gamma(u)=|\Gamma|-|C|$, from which the result follows.

Finally, we prove the upper bound. Suppose the vertices of $\Gamma$ are labeled with elements of $C_m,C_1,\ldots,C_{m-1}$, respectively. Let $M_1,\ldots,M_p$ be the positive matching decomposition of $\Gamma$ obtained from the restriction of the pmd of $K_{|\Gamma|}$, the complete graph on $V(\Gamma)$, to $\Gamma$ as in the proof of Proposition \ref{pmd of K_n} where $p=2|\Gamma|-3$. Then the first $n_m-1$ and the last $n_{m-1}-1$ given positive matchings are contained in the subgraphs induced by $C_m$ and $C_{m-1}$, respectively. Since $C_m$ and $C_{m-1}$ are independent sets, by removing these $(n_m-1)+(n_{m-1}-1)$ positive matchings, we obtain a positive matching decomposition for $\Gamma$, hence $\pmd(\Gamma)\leq2|\Gamma|-n_{m-1}-n_m-1$. The proof is complete.
\end{proof}
%--------------------------------------------------
\begin{corollary}
For all $m,n$, we have
\[\pmd(K_{1,m,n})=m+n+1\]
and
\[\pmd(K_{2,m,n})=m+n+2+\varepsilon,\]
where $\varepsilon\in\{0,1\}$.
\end{corollary}
%--------------------------------------------------
\begin{remark}
A computation with computer reveals that $\pmd(K_{2,2,2})=6$ and $\pmd(K_{2,2,3})=7$. This leads us to the conjecture that $\pmd(K_{2,m,n})=m+n+2$, for all $m,n$. However, the lower bound in the above theorem shows, $\pmd(K_{a,b,c})=a+b+c$ does not hold in general.
\end{remark}

In the rest of this section, we shall study bipartite graphs in more details and give an explicit and algorithmic upper bound for the pmd of bipartite graphs.
%==================================================
\subsection{Regular bipartite graphs}
In the following, we aim to find an upper bound for the pmd of regular bipartite graphs. To this end, we apply the following fundamental theorem of Jacob and Meyniel on digraphs taking maximum dichromatic numbers. Recall that the \textit{dichromatic number} $\overset{\rightarrow}{\chi}(\Gamma)$ of a digraph $\Gamma$ is the minimum number of colors required to assign to vertices of $\Gamma$ such that every color class induces an acyclic disubgraph.
%--------------------------------------------------
\begin{theorem}[Jacob and Meyniel \cite{hj-hm}]\label{dichromatic number}
Let $\Gamma$ be a connected digraph of maximum valency $2r$. Then $\overset{\rightarrow}{\chi}(\Gamma)\leq r$ except in the following cases:
\begin{itemize}
\item[\rm (i)]$r=1$ and $\Gamma$ is a directed circuit;
\item[\rm (ii)]$r=2$ and $\Gamma$ is an odd cycle (with two sided edges); and
\item[\rm (iii)]$\Gamma$ is the complete graph of order $r+1$ (with two sided edges).
\end{itemize}
\end{theorem}
%--------------------------------------------------
\begin{theorem}\label{pmd of bipartite graphs with perfect matching}
Let $\Gamma$ be a bipartite graph with a perfect matching $M$. If $\Delta(\Gamma)\geq3$, then
\[\pmd(\Gamma)\leq\pmd(\Gamma - M)+\Delta(\Gamma)-1.\]
\end{theorem}
\begin{proof}
Let $\Delta:=\Delta(\Gamma)$. If $\Gamma$ is a complete bipartite graph, then $\Gamma\cong K_{\Delta,\Delta}$, and hence 
\[\pmd(\Gamma)=2\Delta-1\leq\pmd(\Gamma- M)+\Delta(\Gamma)-1\]
by Proposition \ref{pmd of K_m,n} (note that $\pmd(\Gamma-M)\geq\Delta$ for $\Delta(\Gamma-M)=\Delta-1$ and $\Gamma-M$ has no positive perfect matching). Hence, we can assume that $\Gamma$ is not a complete bipartite graph. Let $U=\{u_1,\ldots,u_n\}$ and $V=\{v_1,\ldots,v_n\}$ denote a bipartition of $\Gamma$, and assume that $M=\{u_1v_1,\ldots,u_nv_n\}$ is a perfect matching of $\Gamma$. Let $\Gamma'$ be the digraph with vertex set $\{w_1,\ldots,w_n\}$ and arcs $(w_i,w_j)$ if $\{u_i,v_j\}$ is an edge in $\Gamma$, for all $1\leq i\neq j\leq n$. It is evident that $\Gamma'$ is a digraph with maximum valency $2(\Delta-1)$. By Theorem \ref{positivity of matchings}, there is a bijection between alternating closed walks of $\Gamma$ with respect to the matching $M$ and directed closed walks of $\Gamma'$ (compare to \cite[Lemma 5.2(3)]{ac-vw}). If $\Gamma'$ is an odd cycle, say $w_1,\ldots,w_n$, then replacing the matching $M$ with $\{u_1v_2,\ldots,u_{n-1}v_n,u_nv_1\}$, the resulting graph $\Gamma'$ will not be an odd cycle. Hence, we can assume that $\Gamma'$ is not an odd cycle. Now, by Theorem \ref{dichromatic number}, the vertex set of $\Gamma'$ can be partitioned into at most $\Delta-1$ sets inducing acyclic disubgraphs. Therefore, $M$ is a union of at most $\Delta-1$ positive matchings, from which the result follows.
\end{proof}
%--------------------------------------------------
\begin{corollary}\label{pmd of union of two bipartite graphs}
Let the bipartite graph $\Gamma$ be a union of two edge-disjoint spanning subgraphs $\Gamma_1$ and $\Gamma_2$, where $\Gamma_1$ is $r$-regular and $\Gamma_2$ is non-empty with maximum valency $s$. Then
\[\pmd(\Gamma)\leq\pmd(\Gamma_2)+\binom{r}{2}+rs.\]
\end{corollary}
\begin{proof}
K\"{o}nig's theorem \cite[Corollary 2.1.3]{rd} states that every regular bipartite graphs has a perfect matching. It follows that $\Gamma_1$ has a $1$-factorization, that is the edge set of $\Gamma_1$ is partitioned into $r$ perfect matchings, say $M_1,\ldots,M_r$. On the other hand, by Theorem \ref{pmd of bipartite graphs with perfect matching},
\[\pmd(\Gamma - M_1\cup\cdots\cup M_{i-1})\leq\pmd(\Gamma - M_1\cup\cdots\cup M_i)+\Delta(\Gamma)-i,\]
for all $i=1,\ldots,r$. Thus
\begin{align*}
\pmd(\Gamma)&\leq\pmd(\Gamma-M_1)+\Delta(\Gamma)-1\\
&\hspace{0.2cm}\vdots\\
&\leq\pmd(\Gamma - M_1\cup\cdots\cup M_r)+(\Delta(\Gamma)-1)+\cdots+(\Delta(\Gamma)-r)\\
&=\pmd(\Gamma_2)+\binom{r}{2}+rs,
\end{align*}
as required.
\end{proof}
%--------------------------------------------------
\begin{theorem}
For every $r$-regular bipartite graph $\Gamma$, we have $\pmd(\Gamma)\leq\binom{r}{2}+2$.
\end{theorem}
\begin{proof}
The result follows from Corollary \ref{pmd of union of two bipartite graphs}, K\"{o}nig's theorem on the existence of perfect matchings in regular bipartite graphs, and assuming that $\Gamma_2$ is a $2$-factor (union of two disjoint perfect matchings) of $\Gamma$.
\end{proof}
%--------------------------------------------------
\begin{remark}
Since a $3$-regular bipartite graph $\Gamma$ has no positive perfect matchings it follows in conjunction with the above theorem that $4\leq\pmd(\Gamma)\leq5$. We know that $\pmd(K_{3,3})=5$. As we show in the next section, $\pmd(CL_4)=5$ as well (Propositions \ref{pmd of CL_n}). Also, Propositions \ref{pmd of M_n} and \ref{pmd of generalized petersen graphs} yields an infinite family of $3$-regular bipartite graphs with pmd equal to $4$.
\end{remark}
%==================================================
\subsection{An algorithmic approach for approximating pmd of bipartite graphs}
In this part, we present a (possibly efficient) simple algorithm for computing a positive matching decomposition of bipartite graphs. Our algorithm is based on the notion of slopes of edges defined below and has complexity $O((m+n)^2mn)$ for a bipartite graph with part sizes $m$ and $n$. For the source code of the algorithm in GAP see \cite{mfdg}.
%--------------------------------------------------
\begin{definition}
Assume that $\Gamma$ is a bipartite graph with bipartition $X=\{x_1, \ldots, x_m\}$ and $Y=\{y_1, \ldots, y_n\}$. The \textit{slope} of an edge $x_iy_j$ is defined as $s(x_iy_j)=j-i$.
\end{definition}
%--------------------------------------------------
\begin{algorithm}[H]
\begin{algorithmic}[1]\baselineskip=10pt\relax

\REQUIRE A labeled bipartite graph $ \Gamma $ with bipartition $ X, Y $
\ENSURE A positive matching decomposition of $ \Gamma $
\STATE $ k\longleftarrow 0 $
\WHILE{$ E(\Gamma) \neq \varnothing $}
\STATE $ k \longleftarrow k + 1 $
\STATE  $ X' \longleftarrow X \setminus \{\text{isolated vertices of } X\}$
\STATE  $ Y' \longleftarrow \varnothing $
\STATE $M_k \longleftarrow \varnothing $
\WHILE{$ X' \neq \varnothing$}
\STATE $ s \longleftarrow \min(\text{Slopes}(\Gamma[X'\cup Y])) $
\STATE $ M \longleftarrow \{xy \in E(\Gamma[X'\cup Y]) \colon \quad \mathrm{slope}(xy)=s\} $
\STATE $ M_k \longleftarrow M_k \cup \{xy \in M \colon \quad y \notin Y'\} $
\STATE $ X' \longleftarrow X' \setminus \{x \colon \quad xy\in M\} $
\STATE $ Y' \longleftarrow Y' \cup \{y \colon \quad xy\in M\} $
\ENDWHILE
\STATE $E(\Gamma) \longleftarrow E(\Gamma) \setminus M_k $
\ENDWHILE
\RETURN $ M_1,\ldots,M_k$

\end{algorithmic}
\caption{\textsc{A positive matching decomposition of a labeled bipartite graph}}
\label{alg:A positive matching decomposition of a labeled bipartite graph}
\end{algorithm}
%--------------------------------------------------
\begin{theorem}\label{pmd of bipartite graphs by slopes}
Let $\Gamma$ be a bipartite graph. Then the output of the Algorithm~\ref{alg:A positive matching decomposition of a labeled bipartite graph} is a positive matching decomposition of $\Gamma$.
\end{theorem}
\begin{proof}
Let $X=\{x_1,\ldots,x_m\}$ and $Y=\{y_1,\ldots,y_n\}$ be a bipartition of $\Gamma$. It is enough to show that $M_1$ is a positive matching of $\Gamma$. Let $M$ be a non-empty subset of $M_1$ and let $x_sy_t\in M$ be such that $t := \max\{j \colon\ x_iy_j\in M\}$. Clearly, $x_s$ is a pendant vertex in $\Gamma[M]$, otherwise $x_s$ is adjacent to a vertex $y_{t'}$ with $t' < t$, which is impossible by Algorithm~\ref{alg:A positive matching decomposition of a labeled bipartite graph}. Therefore, by Theorem \ref{positivity of matchings}(iii), $M_1$ is a positive matching of $\Gamma$, as required.
\end{proof}
%--------------------------------------------------
\begin{definition}
Let $\Gamma$ be a bipartite graph with bipartition $X$ and $Y$, and $\mathfrak{K}(\tau)$ be the number of matchings in the output of Algorithm~\ref{alg:A positive matching decomposition of a labeled bipartite graph} for every labeling $\tau$ of the vertex set of $\Gamma$. Let
\[ \mathfrak{K}(\Gamma) := \min \{ \mathfrak{K}(\tau) \colon \; \tau \text{ is a labelling} \} \]
be the minimum size of the output of Algorithm~\ref{alg:A positive matching decomposition of a labeled bipartite graph} in any labeling of $\Gamma$.	 Note that $|X| + |Y| - 1$, the maximum possible number of slopes, is a trivial upper bound for $\mathfrak{K}(\Gamma)$.
\end{definition}
%--------------------------------------------------
\begin{example}\ 
\begin{itemize}
\item[(i)]Consider the cyclic graph $\Gamma$ in Fig. \ref{figure: bipartite graph} (left). With this labeling $\Gamma$ has four slopes $-2$, $-1$, $1$, and $2$ but $\mathfrak{K}(\tau)=3$ as Algorithm~\ref{alg:A positive matching decomposition of a labeled bipartite graph} yields the pmd $M_1=\{x_3y_1,x_1y_2\}$, $M_2=\{x_2y_1,x_3y_2,x_1y_3\}$, and $M_3=\{x_2y_3\}$. Hence $\mathfrak{K}(\tau)$ can be smaller than the number of slopes.
\item[(ii)]Let $\Gamma$ be the graph in Fig. \ref{figure: bipartite graph} (right). With this labeling we have $\mathfrak{K}(\tau)=4$ for Algorithm~\ref{alg:A positive matching decomposition of a labeled bipartite graph} yields the pmd $M_1=\{x_3y_1,x_2y_3\}$, $M_2=\{x_1y_1,x_3y_3\}$, $M_3=\{x_1y_2\}$, and $M_4=\{x_1y_3\}$. Note that $\pmd(\Gamma)=3$. Hence $\mathfrak{K}(\tau)$ need not be equal to $\pmd(\Gamma)$ in general.
\end{itemize}
\end{example}
%--------------------------------------------------
\begin{center}
\begin{figure}[!htp]
\begin{tabular}{ccc}
\begin{tikzpicture}
\node [draw, fill=white, circle, inner sep=2pt, label=left:{$x_3$}] (x3) at (0,0) {};
\node [draw, fill=white, circle, inner sep=2pt, label=left:{$x_2$}] (x2) at (0,1) {};
\node [draw, fill=white, circle, inner sep=2pt, label=left:{$x_1$}] (x1) at (0,2) {};
\node [draw, fill=white, circle, inner sep=2pt, label=right:{$y_3$}] (y3) at (1,0) {};
\node [draw, fill=white, circle, inner sep=2pt, label=right:{$y_2$}] (y2) at (1,1) {};
\node [draw, fill=white, circle, inner sep=2pt, label=right:{$y_1$}] (y1) at (1,2) {};

\draw (x1)--(y2)--(x3)--(y1)--(x2)--(y3)--(x1);
\end{tikzpicture}
&\ \hspace{2cm}\ &
\begin{tikzpicture}
\node [draw, fill=white, circle, inner sep=2pt, label=left:{$x_3$}] (x3) at (0,0) {};
\node [draw, fill=white, circle, inner sep=2pt, label=left:{$x_2$}] (x2) at (0,1) {};
\node [draw, fill=white, circle, inner sep=2pt, label=left:{$x_1$}] (x1) at (0,2) {};
\node [draw, fill=white, circle, inner sep=2pt, label=right:{$y_3$}] (y3) at (1,0) {};
\node [draw, fill=white, circle, inner sep=2pt, label=right:{$y_2$}] (y2) at (1,1) {};
\node [draw, fill=white, circle, inner sep=2pt, label=right:{$y_1$}] (y1) at (1,2) {};

\draw (y2)--(x1)--(y1)--(x3)--(y3)--(x1) (x2)--(y3);
\end{tikzpicture}\\
\end{tabular}
\caption{}
\label{figure: bipartite graph}
\end{figure}
\end{center}

From Theorem \ref{pmd of bipartite graphs by slopes}, we obtain the following result immediately.
%--------------------------------------------------
\begin{corollary}
Let $\Gamma$ be a bipartite graph. Then 
\[ \Delta (\Gamma) \leq \pmd(\Gamma) \leq \mathfrak{K}(\Gamma). \]
\end{corollary}

We conclude this subsection with two related questions.
%--------------------------------------------------
\begin{question}\label{pmd(G)=K(G)?}
Is there any bipartite graph $\Gamma$ satisfying $\pmd(\Gamma)<\mathfrak{K}(\Gamma)$?
\end{question}
%--------------------------------------------------
\begin{question}\label{K(G)=s?}
Is there any bipartite graph $\Gamma$ and a labeling $\tau$ for which $\mathfrak{K}(\tau)<s$ with $\tau$ having the minimum number of slopes $s$ among all possible labellings of $\Gamma$?
\end{question}
%==================================================
\section{Some families of graphs and their pmds}\label{section 4}
In this section, we shall compute the pmd of several classes of graphs. Indeed, we compute the pmds of generalized Petersen graphs and M\"{o}bius ladders while we only give an upper bound for the pmd of hypercubes.

First, we consider the circular ladders $CL_n:=GP(n,1)$ ($n\geq3$), which is the Cartesian product of an $n$-cycle and an edge.
%--------------------------------------------------
\begin{proposition}\label{pmd of CL_n}
We have $\pmd(CL_n)=4$ for $n\neq4$, and $\pmd(CL_n)=5$ for $n=4$.
\end{proposition}
\begin{proof}
Let $\Gamma:=CL_n$ with vertices $u_1,\ldots,u_n,v_1,\ldots,v_n$, where $u_i\sim u_{i+1}$, $v_i\sim v_{i+1}$, and $u_i\sim v_i$, for $i=1,\ldots,n$. Notice that all indices are taken modulo $n$. First we show that $\pmd(\Gamma)\geq4$. Let $M$ be the first positive matching in a minimum sized pmd of $\Gamma$, and $\Gamma'$ be the graph obtained from $\Gamma$ by removing edges in $M$. If $M$ contains an spoke $u_iv_i$, then none of the edges of the cycle $u_i,u_{i+1},v_{i+1},v_i,v_{i-1},u_{i-1}$ belongs to $M$ so that $\pmd(\Gamma)=\pmd(\Gamma')+1\geq4$. Hence, assume that $M$ does not have any spoke. As $M$ has less than $n/2$ edges from any of the cycles induced by $\{u_1,\ldots,u_n\}$ and $\{v_1,\ldots,v_n\}$ by Example \ref{simple pmds}(iii), there exists $1\leq i\leq n$ such that $u_iu_{i+1},v_iv_{i+1}\notin M$. Thus none of the edges of the cycle $u_i,u_{i+1},v_{i+1},v_i$ belongs to $M$ so that $\pmd(\Gamma)=\pmd(\Gamma')+1\geq4$ in this case too.

Now we show that $\Gamma$ has a pmd of size four for $n\neq4$. It is easy to see that
\begin{align*}
M_1&=\{u_1v_1,u_3v_3,\ldots,u_{n-2}v_{n-2}\}\cup\{v_{n-1}v_n\},\\
M_2&=\{u_2v_2,u_4v_4,\ldots,u_{n-1}v_{n-1}\}\cup\{u_nu_1\}%,\\
\end{align*}
if $n$ is odd, and 
\begin{align*}
M_1&=\{u_1v_1,u_3v_3,\ldots,u_{n-1}v_{n-1}\},\\
M_2&=\{u_2v_2,u_4v_4,\ldots,u_{n-2}v_{n-2}\}\cup\{v_{n-1}v_n,u_nu_1\}%,\\
\end{align*}
if $n$ is even are positive matchings of $\Gamma$, respectively. Since $\Gamma-M_1\cup M_2$ is a union of paths, it follows that $\pmd(\Gamma)=4$. For $n=4$, we may show by direct computation that $\pmd(CL_4)=5$ with a pmd of size $5$ as given in Proposition \ref{pmd of Q_n}.
\end{proof}

A similar argument, with a little modification, establishes the result for M\"{o}bius ladders.
%--------------------------------------------------
\begin{proposition}\label{pmd of M_n}
We have $\pmd(M_n)=4$, for all $n\geq3$.
\end{proposition}
\begin{proof}
Let $\Gamma:=M_n$ with vertices $u_1,\ldots,u_n,v_1,\ldots,v_n$, where $u_i\sim u_{i+1}$, $v_i\sim v_{i+1}$, for $i=1,\ldots,n-1$, $u_n\sim v_1$, $v_n\sim u_1$, and $u_i\sim v_i$, for $i=1,\ldots,n$. As in Proposition \ref{pmd of CL_n}, we have $\pmd(\Gamma)\geq4$. Let
\begin{align*}
N_1&=\{u_2v_2,u_4v_4,\ldots,u_{n-1}v_{n-1}\}\cup\{u_1v_n\},\\
N_2&=\{u_1v_1,u_3v_3,\ldots,u_uv_n\}%,\\
\end{align*}
if $n$ is odd, and
\begin{align*}
N_1&=\{u_2v_2,u_4v_4,\ldots,u_{n-2}v_{n-2}\}\cup\{u_1v_n,u_{n-1}u_n\},\\
N_2&=\{u_3v_3,u_5v_5,\ldots,u_{n-1}v_{n-1}\}\cup\{u_nv_1,u_1u_2\}%,\\
\end{align*}
if $n$ is even. Then $N_1$ is a positive matching in $\Gamma$, $N_2$ is a positive matching in $\Gamma-N_1$, and $\Gamma-N_1\cup N_2$ is a union of paths, which imply that $\pmd(\Gamma)=4$.
\end{proof}
%--------------------------------------------------
\begin{proposition}\label{pmd of generalized petersen graphs}
Let $\Gamma$ be the generalized Petersen graph $GP(n,k)$ ($1\leq k\leq n/2$). Then 
\[\pmd(\Gamma)=\begin{cases}
3,&n=2k\ \emph{and}\ k\neq3,\\
5,&n=8\ \emph{and}\ k=1,\\
4,&\emph{otherwise}.\end{cases}\]
\end{proposition}
\begin{proof}
Assume the vertices of $\Gamma$ are labeled with $u_i,v_i$, for $i\in\mathbb{Z}_n$ where $u_i\sim u_{i+1}$, $v_i\sim v_{i+k}$, and $u_i\sim v_i$, for all $i\in\mathbb{Z}_n$. First observe that $\Gamma$ is the circular ladder when $k=1$ and its pmd is computed in Proposition \ref{pmd of CL_n}. Assume $k>1$ and $n=2k$ is even. A simple computation yields $\pmd(\Gamma)=4$ when $k=3$. If $k=2k'$ is even, then 
\[\{u_{2i-1}u_{2i}\colon\quad 1\leq i\leq k'\}\cup\{u_{k+1}v_{k+1},\ldots,u_nv_n\}\]
is a positive matching in $\Gamma$ whose removal result in a union of paths. Thus $\pmd(\Gamma)=3$ when $k$ is even. Now, assume that $k>3$ is odd. Then 
\[\{u_1v_1,u_2u_3,u_4v_4\}\cup\{u_{2i-1}u_{2i}\colon\quad 3\leq i\leq k\}\]
is a positive matching whose removal is a Hamilton path. Thus $\pmd(\Gamma)=3$ in this case too.

In what follows, we assume that $2\leq k<n/2$. Let $p:=\pmd(\Gamma)$ and $M_1,\ldots,M_p$ be a pmd of $\Gamma$. Then $\pmd(\Gamma)=\pmd(\Gamma-M_1)+1\geq4$ for $\delta(\Gamma-M_1)\geq2$ so that $\Gamma-M_1$ has a cycle. Indeed, we show that $\pmd(\Gamma)=4$.

First assume that $\gcd(n,k)=1$, that is there is only one inner cycle. Let 
\[M_1:=\{v_1v_{k+1}\}\cup\{u_{2i}v_{2i}\colon\quad 1\leq i\leq k/2\}\cup\{u_{k+2i}v_{k+2i}\colon\quad 1\leq i\leq (n-k)/2\}\]
and
\[M_2:=\{u_1u_2\}\cup\{u_{2i+1}v_{2i+1}\colon\quad 1\leq i<k/2\}\cup\{u_{k+2i-1}v_{k+2i-1}\colon\quad 1\leq i\leq(n-k+1)/2\}.\]
Utilizing Theorem \ref{positivity of matchings}(iii), one observe simply that $M_1$ is a positive matching in $\Gamma$. Applying Theorem \ref{positivity of matchings}(iii) once more, one can also show the positivity of $M_2$ in $\Gamma-M_1$. Since $k>1$, we observe that $\Gamma-M_1\cup M_2$ is a Hamilton path. Thus $\pmd(\Gamma)\leq4$ so that $\pmd(\Gamma)=4$.

Now, assume that $d:=\gcd(n,k)>1$. Then we have $d$ inner cycles including vertices $v_1,\ldots,v_d$, respectively. Let 
\begin{align*}
M_1:=\{v_iv_{i+k}\colon\quad 1\leq i\leq d\}&\cup\{u_{d+2i-1}v_{d+2i-1}\colon\quad 1\leq i\leq (k-d+1)/2\}\\
&\cup\{u_{d+k+2i-1}v_{d+k+2i-1}\colon\quad 1\leq i\leq (n-k-d+1)/2\}
\end{align*}
and
\begin{align*}
M_2:=\{u_iv_i\colon\quad 1\leq i\leq d\}&\cup\{u_{d+2i}v_{d+2i}\colon\quad 1\leq i\leq (k-d)/2\}\\
&\cup\{u_{k+i}v_{k+i}\colon 1\leq i\leq d-2\}\cup\{u_{d+k-1}u_{d+k}\}\\
&\cup\{u_{d+k+2i}v_{d+k+2i}\colon\quad 1\leq i\leq (n-k-d)/2\}.
\end{align*}
As in the case where $n$ and $k$ are coprime we can see that $M_1$ is a positive matching in $\Gamma$ and $M_2$ is a positive matching in $\Gamma-M_1$. Also, $\Gamma-M_1\cup M_2$ is a union of $d-1$ paths. Thus $\pmd(\Gamma)\leq4$, which implies that $\pmd(\Gamma)=4$. The proof is complete.
\end{proof}

Finally, we consider the hypercubes. Recall that for a group $G$ and an inversed-closed subset $C$ of $G\setminus\{1\}$, the \textit{Cayley graph} $\Cay(G,C)$ is a graph with vertex set $G$ such that two vertices $x,y\in G$ are adjacent if $yx^{-1}\in C$.	
%--------------------------------------------------
\begin{proposition}\label{pmd of Q_n}
We have $\pmd(Q_n)\leq 2n-1$, for all $n\geq1$.
\end{proposition}
\begin{proof}
First observe that $|V(Q_n)|=2^n$ and $|E(Q_n)|=n2^{n-1}$, for all $n\geq1$. Also, $Q_n=\Cay(\mathbb{Z}_2^n,\{\mathbf{e}_1,\ldots,\mathbf{e}_n\})$, where $\mathbf{e}_i$ is the element with $1$ in the $i$th entry and $0$ elsewhere. We set $\varepsilon(\mathbf{e}_i):=i$, for all $i=1,\ldots,n$. For every edge $e=uv\in E(Q_n)$, we define the subgroup $G_e$ of $\mathbb{Z}_2^n$ as the group of all elements of $\mathbb{Z}_2^n$ with zero coordinate sum and zero $\varepsilon(u-v)$-entry. A simple verification shows that $|G_e|=2^{n-2}$ and $Q_n[e+G_e]\cong 2^{n-2}K_2$ is a matching, so that $e+G_e$ is a positive matching in $Q_n$ by Theorem \ref{positivity of matchings}. Let $H$ be the maximal subgroup of $\mathbb{Z}_2^n$ consisting of elements with zero $n$th entry. It is not difficult to see that the sets $e+G_e$ partition $E(Q_n)$. Also
\[e+G_e\subseteq E(Q_n[H])\cup E(Q_n[H+\mathbf{e}_n]),\]
for all $e\in E(Q_n[H])$. Hence, there exist edges $e_1,\ldots,e_{2n-2}\in Q_n[H]$ for which $e_1+G_{e_1},\ldots,e_{2n-2}+G_{e_{2n-2}}$ partition $E(Q_n[H])\cup E(Q_n[H+\mathbf{e}_n])$. Then 
\[e_1+G_{e_1},\ldots,e_{2n-2}+G_{e_{2n-2}},\{\{h,h+\mathbf{e}_n\}\colon\ h\in H\}\]
determine a pmd of $Q_n$, from which the result follows.
\end{proof}
%--------------------------------------------------
\begin{conjecture}
We have $\pmd(Q_n)=2n-1$, for all $n\geq1$.
\end{conjecture}

%==================================================

\end{document}